\documentclass[uplatex,dvipdfmx.]{article}
\usepackage{amsmath,amssymb,verbatim,amsthm}
\usepackage[all]{xy}
\usepackage{titlesec}
\usepackage[margin=22.5truemm]{geometry}
\usepackage{indentfirst}

\newcommand{\C}{\mathbb{C}}

\newcommand{\N}{\mathbb{N}}

\newcommand{\cF}{\mathcal{F}}
\newcommand{\cG}{\mathcal{G}}

\newcommand{\cI}{\mathcal{I}}

\newcommand{\cO}{\mathcal{O}}

\renewcommand{\leq}{\leqslant}

\newcommand{\abs}[1]{\left\lvert#1\right\rvert}
 
\newcommand{\ip}[2]{\langle #1,#2\rangle}

\renewcommand{\Im}{\mathrm{Im}}

\renewcommand{\Re}{\mathrm{Re}}

\newcommand{\kahler}{k\ddot{a}hler}
\newcommand{\hormander}{H\ddot{o}rmander}

\DeclareMathOperator{\codim}{codim}

\DeclareMathOperator{\imaginary}{\sqrt{-1}}
\DeclareMathOperator{\dd}{\sqrt{-1}\partial\bar{\partial}}
\DeclareMathOperator{\dbar}{\bar{\partial}}
\DeclareMathOperator{\diag}{diag}

\DeclareMathOperator{\loc}{loc}

\DeclareMathOperator{\Tr}{Tr}

\DeclareMathOperator{\vol}{vol}
\DeclareMathOperator{\Wedge}{\bigwedge}

\numberwithin{equation}{section}       

\newtheorem{prop} {Proposition} [section]
\newtheorem{thm}[prop] {Theorem} 
\newtheorem{main thm}[prop] {Main Theorem}
 
\newtheorem{dfn}[prop] {Definition}
\newtheorem{lem}[prop] {Lemma}
\newtheorem{cor}[prop]{Corollary}

\newtheorem{exam}[prop]{Example}
\newtheorem{rem}[prop]{Remark}
\newtheorem{question}[prop]{Question}

\newtheorem*{thmA*}{\bf{Theorem A}} 
\newtheorem*{thmB*}{\bf{Theorem B}} 
\newtheorem*{thmC*}{\bf{Theorem C}} 
\newtheorem*{thmD*}{\bf{Theorem D}} 
\newtheorem*{thmE*}{\bf{Theorem E}}

\newtheorem{dfn*}{\bf{Definition}}

\begin{document}

\title{\Large{On the sharp $L^2$-estimates of Skoda division theorem} }
\author{Masakazu Takakura}
\date{}

\maketitle
\begin{abstract}
In this paper, we prove a Skoda type division theorem with sharp $L^2$-estimate. Furthermore, using this estimate, we provide new characterizations of plurisubharmonic functions. We also explain that the sharp $L^2$-division theorem leads the Guan-Zhou's sharp $L^2$-estimate for extension theorem.
\end{abstract}
\tableofcontents
\section{Introduction}
\subsection[Background]{Background}
Let $\Omega$ be a domain in $\C^n$. Let $f, g_1, \dots, g_r$ be holomorphic functions on $\Omega$. The classical division problem askes the question of when holomorphic functions $\{F_i\}_{i=1} ^{r}$ on $\Omega$ exist such that $$f = \sum g_i F_i.$$

The first result on this problem was obtained by Oka(\cite{Oka50}): if $\Omega$ is a psuedoconvex domain and $f_x \in \sum g_{i,x} \cO_x $ for any $x\in\Omega$, then a global solution $\{F_i\} \in \cO^{\oplus r}(\Omega)$ exist. This is the one of important results in correspondences between solvability of holomorphic equations and convexities. Oka's division theorem is an algebraically perfect criterion, but determining the existence of local solutions is inherently difficult.
Subsequently, Skoda obtained the following $L^2$-effective criterion of the division problem based on $\rm{\hormander}$'s $L^2$-existence theorem for $\dbar$-equation.
\begin{thm}[\cite{Skoda72}]\label{Skoda division}
    Let $\Omega$ be a psuedoconvex domain in $\mathbb{C}^n$ and $\phi$ be a plurisubharmonic function on $\Omega$. Let  $g = (g_1,\dots,g_r)$ be non-zero $r$-tuple holomorphic functions and put $q = \min(n,r-1)$. Then for any holomorphic function $f$ satisfying 
    \begin{equation}
        \int_{\Omega}\abs{f}^2 \abs{g}^{-2(q+1+\epsilon)} e^{-\phi} d\lambda < +\infty
    \end{equation}
    for lebesgue measure $d\lambda$, there exists a tuple of holomorphic function $F = (F_1,\dots,F_r)$ satisfying
    \begin{gather}
        \sum_{i=1}^{r} F_{i} g_i = f \;\;\; and\\
        \int_{\Omega} \abs{F}^2 \abs{g}^{-2(q+\epsilon)} e^{-\phi} d\lambda \leq(1+q/\epsilon) \int_{\Omega} \abs{f}^2 \abs{g}^{-2(q+1+\epsilon)} e^{ -\phi} d\lambda.
    \end{gather}
\end{thm}
This theorem implies that if $\abs{f}^2 \abs{g}^{-2(q+1)}$ is locally integrable around $x \in \Omega$ then $f_x \in \sum g_{i,x} \cO_x $. As an applications of this theorem (or more general results\cite{Sko 78}), significant results in complex analysis and complex geometry are known, such as the Skoda-Briancon theorem(\cite{Skoda 74}) and Siu's deformation invariance of plurigenera(\cite{Si98}). Furthermore, various generalization or new proof of Skoda's $L^2$ division theorem are known(\cite{A24},\cite{O04},\cite{Sko 78},\cite{V08}). We are interested in whether these estimate are optimal. One can see that the Skoda's original $L^2$-estimate includes the following equality hold case.
\begin{exam}\label{equality of original skoda}
    We consider the division problem $F_0 + \sum_{i=1}^{n} F_i z_i =1$ on $\mathbb{C}^n$. Then $F = (1,0,\dots ,0)$ is the $L^2 (\mathbb{C}^n, (1+\abs{z}^2)^{-(n+\epsilon)})$ minimum solution. Then the following equality holds: 
    $$\int_{\mathbb{C}^n} \abs{F}^2 (1+\abs{z}^2)^{-(n+\epsilon)} =  \int_{\mathbb{C}^n} (1+\abs{z}^2)^{-(n+\epsilon)} =(1+n/\epsilon) \int_{\mathbb{C}^n} 1\cdot(1+\abs{z}^2)^{-(n+1+\epsilon)}$$
\end{exam} 
However, the optimal estimate of $L^2$-division theorem on bounded psuedoconvex domain has not been obtained so far. On the other hand, regarding the $L^2$-extension theorem, several optimal results are known. The first was obtained by Blocki (\cite{blocki suita}). He showed optimal $L^2$-extension theorem on psuedoconvex domains which includes resolution of Suita's conjecture. Later, a more general results was obtained by Guan and Zhou (\cite{GZ12},\cite{GZ15}). Their proof relied on making the best use of the twisted $\dbar$ estimates by Ohsawa-Takegoshi(\cite{OT87}). Motivated by these recent development of $\bar{\partial}$-estimate techniques, purpose of this paper is to revisit the $L^2$-estimate of division theorem.
Our first resuls establishes the division theorem with sharp $L^2$-estimate by using Guan and Zhou's $\dbar$-estimate technique. The inequality estimate in our theorem includes the cases of equality in more natural settings than Example \ref{equality of original skoda} (see Example \ref{sharpness of estimate}). In section \ref{Characterizing plurisubharmonicity via the sharp L^2-division property}, we discuss the inverse problem of whether the sharp estimate of division theorem can characterize plurisubharmonic function. This is inspired by \cite{DNWZ23},\cite{DWZ18}.
In section \ref{A proof of sharp L2 extension theorem}, we further explain how the sharp $L^2$-division theorem leads the sharp $L^2$-extension theorem. 
In \cite{A25}, it is shown that a twisted $L^2$-estimate of division theorem leads an $L^2$-extension theorem. By replacing this twisted estimate with the sharp estimate, we recover the Guan-Zhou's sharp estimates.

\par
\subsection[Main Theorem]{Main Theorem}
Let $E$ and $Q$ be holomorphic vector bundles of rank $r_E$ and $r_Q$ respectively on an $n$-dimensional weakly psuedoconvex $\kahler$ manifold $(X,\omega)$. Let $h_E$ and $e^{-\phi}$ be smooth hermitian metrics of $E$ and $\det{Q}$, $\Theta_E$ be the chern curvature of $(E,h)$. Let $$g \colon E \rightarrow Q$$ be a generically surjective bundle morphism, $e^{-\tilde{\phi}}$ be a (possibly singular) hermitian metric of $\det{Q}$ induced by $g$. Put $$\Phi = \tilde{\phi} -\phi$$ as a well-defined quasi plurisubharmonic function on $X$, for integer $0\le k \le n$, $q=\min{(n-k,r_E - r_Q)}$.

\begin{thmA*}\label{sharp L^2 division}
   Assume that there exists a triple $(C,D,S) \in \cG_{\Phi}$ (see {\rm{Definition \ref{def of gain}}} bellow) such that

    \begin{gather}
        \imaginary\Theta_E \ge_q 0 \label{condition 1}\\
        \left(q r_Q +1 + \frac{r_Q}{S(\Phi)} \right)\imaginary\Theta_E \ge_q \left(q+\frac{1}{S(\Phi)}\right) \dd \phi \otimes Id_E \label{condition 2}.
    \end{gather}
     Then for any $\dbar$-closed $(n,k)$-form $f$ with $L^2_{\loc}$ coefficients taking value in $Q$ such that 
    \begin{equation}
        \int_X \left(C(\Phi) + qD (\Phi)\right) \abs{g^* f}^2_{h_E,\omega} e^{-q\Phi} dV_{\omega} < +\infty ,
    \end{equation}
there exists a $\dbar$-closed $(n,k)$-form $F$ with $L^2_{\loc}$ coefficients taking value in $E$ such that $g(F) = f$ and 
    \begin{gather}
        \int_X C(\Phi) \abs{F}_{h_E,\omega} ^2 e^{-q\Phi} dV_{\omega} \leq \int_X \left(C(\Phi) + qD (\Phi)\right)\abs{g^* f}^2_{h_E,\omega} e^{-q\Phi} dV_{\omega}.
    \end{gather}    
\end{thmA*}
\begin{dfn}\label{def of gain}
    Let $I \subset \mathbb{R}$ be an open interval, $\cG_I$ be a set of triple of positive smooth functions $(C,D,S)$ on $I$ which satisfies 
    \begin{gather}\label{ode}
    \frac{d}{dt} S(t) < 0, \;\;
    \frac{d}{dt} D(t) = -C(t), \; \;and \;\;
    \frac{d}{dt} \left(S(t) D(t)\right) = -D(t).
    \end{gather}
 For a function $\Phi$ on X, we define $\cG_{\Phi}=\bigcup_{\Phi\left(X\right) \subset I} \cG_I $. And if a triple $(C,D,S) \in \cG_{(-\infty,0)}$ satisfies $$\int_{0}^{1} r D(\log r^2) dr <+\infty \;\;\;and $$ $$\lim_{r\rightarrow 1} D(r) =0,$$ then we say that $(C,D,S)$ satisfies sharp condition.
\end{dfn}
\begin{rem}\label{basic rem1}
    We can write $S(t) = E(t)/D(t)$ for some positive primitive function $E$ of $-D$. Then 
    \begin{equation}
        \dot{S} = (CE -D^2)/D^2 <0
    \end{equation}
    is equivalent to strict concavity of $\log{E}$. Examples of $(C,D,S) \in \cG_I$ and the method to find them is described in subsection \ref{methods for finding}.
\end{rem}
Applying Theorem A to the classical division problem setting, we get the following result.
\begin{thm}\label{optimal division}
    Let $\Omega$ be a psuedoconvex domain in $\mathbb{C}^n$ and $\phi$ be a plurisubharmonic function on $\Omega$. Let  $g = (g_1,\dots,g_r)$ be a non-zero $r$-tuple holomorphic functions. Put $q = \min{(r-1 ,n)}$, $\Phi = \log \abs{g}^2$, and take $(C,D,S) \in \cG_{\Phi}$. Then for any holomorphic function $f$ satisfying 
    \begin{equation}
        \int_{\Omega} (C(\Phi) + qD (\Phi)) \abs{f}^2 \abs{g}^{-2(q+1)} e^{-\phi} d\lambda < +\infty
    \end{equation}
    for the lebesgue measure $d\lambda$, there exists a tuple of holomorphic function $F = (F_1,\dots,F_r)$ satisfying
    \begin{gather}
        \sum_{i=1}^{r} F_{i} g_i = f \;\;\;and\\
        \int_{\Omega} C(\Phi) \abs{F}^2 \abs{g}^{-2q} e^{-\phi} d\lambda \leq \int_{\Omega} (C(\Phi) + qD (\Phi)) \abs{f}^2 \abs{g}^{-2(q+1)} e^{-\phi} d\lambda.
    \end{gather}  
\end{thm}
Assume $\log{\abs{g}^2}<0$ , and let $(C,D,S)$ be $(e^{-\epsilon t} , \frac{e^{\epsilon t}-1}{\epsilon},\frac{(e^{-\epsilon t} +\epsilon t - 1)}{\epsilon(e^{-\epsilon t} - 1)} )$. Applying Theorem \ref{optimal division}, we get the following estimate which is strictly stronger than Skoda's original estimate.
\begin{cor}\label{refinement of skoda}
    Let $\Omega$ be a psuedoconvex domain in $\mathbb{C}^n$ and $\phi$ be a plurisubharmonic function on $\Omega$. Let  $g = (g_1,\dots,g_r)$ be non-zero $r$-tuple holomorphic functions such that $\abs{g} <1$. Put $q = \min{(r-1 ,n)}$. Then for any holomorphic function $f$ satisfying 
    \begin{equation}
        \int_{\Omega}\abs{f}^2 \abs{g}^{-2(q+1+\epsilon)} e^{-\phi} d\lambda < +\infty
    \end{equation}
    for lebesgue measure $d\lambda$, there exists a tuple of holomorphic function $F = (F_1,\dots,F_r)$ satisfying
    \begin{gather}
        \sum_{i=1}^{r} F_{i} g_i = f \\
        \int_{\Omega} \abs{F}^2 \abs{g}^{-2(q+\epsilon)} e^{-\phi} d\lambda \leq(1+q/\epsilon) \int_{\Omega} \abs{f}^2 \abs{g}^{-2(q+1+\epsilon)} e^{ -\phi} d\lambda -q/\epsilon \int_{\Omega} \abs{f}^2 \abs{g}^{-2(q+1)} e^{-\phi}. d\lambda
    \end{gather}
\end{cor}
\begin{exam}(sharpness of estimate)\label{sharpness of estimate}
    
    \begin{enumerate}
        \item\label{example1} Assume that $(C,D,S) \in \cG_{(-\infty,0)}$ satisfies sharp condition (e.g Take $(C,D,S)$ as in Corollary \ref{refinement of skoda} for $\delta \in (0,1)$).Put $\Omega$ = $\{(z_1,\dots,z_n) \in \mathbb{C}^n | \sum_{i=1}^{n} a_i ^2 \abs{z_i}^2 < 1\}$ for positive numbers $\{a_i\}_i$. For the tuple of holomorphic functions $g=(a_1 z_1,\dots,a_n z_n)$, $f= \sum_{i=1}^{n} a_i z_i$, a tuple of holomorphic functions $F = (1,\dots,1)$ is the $L^2(\Omega,C(\log\abs{g}^2)\abs{g}^{-2(n-1)})$ minimum holomorphic solution of 
        \begin{equation}\sum_{i=1}^{n} a_i z_i F_i = f.
    \end{equation}
    We denote the H$\ddot{o}$usdorff volume of set $\{\abs{g}^2 = 1\}$
    by $\sigma$. Then we have
        \begin{align}
            &\int_{\Omega} C(\log \abs{g}^2) \abs{F}^2 \abs{g}^{-2(n-1)} d\lambda 
            = n\sigma \int_{0}^{1} r C(\log r^2),\\
            &\int_{\Omega}(C(\log \abs{g}^2) + (n-1)D (\log \abs{g}^2)) \abs{f}^2 \abs{g}^{-2n}d\lambda
            =\sigma \int_{0}^{1} r C(\log r^2) + (n-1)\sigma \int_{0}^{1} r D(\log r^2)
        \end{align}
        and
        \begin{align}
            \int_{0}^{1} r D(\log r^2)
            &=\left[ \frac{r^2}{2} D(\log r^2) \right]_{0} ^{1} + \int_{0}^{1} r C(\log r^2)\\
            &= \int_{0}^{1} r C(\log r^2).
        \end{align}
        Thus the following equality holds:
        \begin{align}
            \min_{\substack{F \in \cO^{\oplus n}(\Omega)\\\sum g_i F_i = f}}\int_{\Omega} C(\log \abs{g}^2) \abs{F}^2 \abs{g}^{-2(n-1)}d\lambda =
                \int_{\Omega} (C(\log \abs{g}^2) + (n-1)D (\log \abs{g}^2)) \abs{f}^2 \abs{g}^{-2n}d\lambda.
        \end{align}
\item\label{example2} We put $(C,D,S) = (e^t, 1-e^t, (1-t-e^t)/(1-e^t)) \in \cG_{(-\infty,0)}$. Let $\mathbb{B} = \{z\in \mathbb{C} : \abs{z} <1\}$ be a open unit ball. For $\theta\in (0,\frac{\pi}{2})$, $g = (\cos\theta,z \sin\theta)$ and $f = 1$, a tuple of holomorphic functions $F = (1/\cos\theta,0)$ is the $L^2(\mathbb{B})$--minimum holomorphic solution of $$g(F) = 1.$$ We have
\begin{align*}
    \int_{\mathbb{B}} \abs{F}^2 d\lambda &= \frac{\pi}{\cos^2\theta}\;\;\;and
\end{align*}
\begin{align*}
    \int_{\mathbb{B}} \frac{\abs{f}^2}{(\cos^2\theta+ \abs{z}^2\sin^2\theta)^2} d\lambda &= 2\pi \int_0^1 \frac{r}{(\cos^2\theta + r^2\sin^2\theta)^2}\\
    &= \frac{\pi}{\cos^2\theta}.
\end{align*}
Thus the following equality holds:
        \begin{align}
            \min_{\substack{F \in \cO^{\oplus 2}(\mathbb{B})\\ \cos \theta F_1 + \sin\theta z F_2 = 1}}\int_{\Omega}  \abs{F}^2 d\lambda =
                \int_{\mathbb{B}} \frac{1}{(\cos^2\theta+ \abs{z}^2\sin^2\theta)^2} d\lambda.
        \end{align}
\end{enumerate}
\end{exam}

Corollary \ref{optimal division} and Example \ref{sharpness of estimate} implies that the plurisubharmonicity of $\phi$ ensures the solvability of division problem with sharp $L^2$-estimate, we say that $\phi$ satisfies the sharp $L^2$-division property. It is well known that plurisubharmonic functions have some other $L^2$-property: (sharp) $L^2$-extension property or $L^2$-estimate property. Interestingly, it has recently been turned out that the converse also holds. In other words,  these $L^2$-properties characterize the plurisubharmonicity(\cite{DNWZ23},\cite{DWZ18}). Our second result is that the sharp $L^2$-division property also characterizes the plurisubharmonicity of $C^2$ class functions.

Fix $(C,D,S) \in \cG_{(-\infty,0)}$ satisfying sharp condition.

\begin{dfn}
Let $\Omega$ be a domain in $\mathbb{C}^n$ $(n\ge 2)$ and $\phi$ be a continuous function. We say that $(\Omega,\phi)$ has the sharp $L^2$-division property if it satisfies following: Assume that holomorphic functions $(f,g_1,\dots g_n) \in\cO^{\oplus n}(\Omega)$ satisfie $\log\abs{g} <0$, $dg_1\wedge \dots \wedge dg_n \ne 0$ everywhere and $$\int_{\Omega} (C(\log \abs{g}^2) + qD(\log\abs{g}^2)) \abs{f}^2 \abs{g}^{-2n} e^{-\phi} d\lambda < +\infty.$$ Then there exists holomorphic functions $(F_1,\dots,F_n) \in \cO^{\oplus n}(\Omega)$ such that $$\sum g_i F_i = f\;\;\;and$$ $$\int_{\Omega} C(\log\abs{g}^2) \abs{F}^2 \abs{g}^{-2(n-1)} e^{-\phi}\le \int_{\Omega} (C(\log \abs{g}^2) + qD(\log\abs{g}^2)) \abs{f}^2 \abs{g}^{-2n} e^{-\phi} d\lambda.$$ 

\end{dfn}
\begin{thmB*}\label{converse}
    Let $\Omega$ be a domain in $\mathbb{C}^n (n\ge2)$ and $\phi \in C^2 (\Omega)$.
\begin{enumerate}
    \item Assume that for any psuedoconvex subdomain $\tilde{\Omega}$, tuples $(\tilde{\Omega},\phi)$ have the sharp $L^2$-division property.
    Then $e^{\phi}$ is plurisubharmonic function i.e $\dd\phi + \imaginary \partial\phi \wedge \dbar \phi$ is positive definite, and it imply that $\dd\phi$ has at most one negative eigenvalue.
    \item Assume that for any psuedoconvex subdomain $\tilde{\Omega}$ and any pluriharmonic function $\psi \in PH(\Omega)$, tuples $(\tilde{\Omega},\phi+\psi)$ has the sharp $L^2$-division property.
    Then $\phi$ is plurisubharmonic.
    \item Assume that for any psuedoconvex subdomain $\tilde{\Omega}$ and positive number $\epsilon$, tuples $(\tilde{\Omega},\epsilon \phi)$ has the sharp $L^2$-division property.
    Then $\phi$ is plurisubharmonic.
\end{enumerate}    
\end{thmB*}

In section \ref{A proof of sharp L2 extension theorem}, we explain that the sharp $L^2$-division theorem leads the Guan--Zhou's sharp $L^2$-extension theorem. Here, (\ref{example2}) in Example 1.7 plays a crucial role as an important observation. The idea is essentially based on \cite{A25}. In \cite{A25}, the author gave a new proof of $L^2$-extension theorem by using twisted $L^2$-division theorem. By replacing the twisted $L^2$-division with sharp $L^2$-division theorem, we can get the Guan-Zhou's sharp $L^2$-estimate. We now state the precise form of sharp $L^2$-extension theorem in a simpler special case. 
Let $X$ be an $n$-dimensional stein manifold, and $S$ be a smooth closed submanifold with $\codim S =k \ge 1$. Assume that there exists holomorphic functions $T_i$ $(i = 1,\dots,k)$ such that $$ S = \{T_1 = \dots = T_k = 0\}\;\;\;and$$
$$dT := dT_1 \wedge \dots \wedge dT_k \ne 0\;\;\; on\; S.$$
We also assume that there exists a negative plurisubharmonic function $\Psi$ on $X$ such that $\Psi \in C^{\infty}(X/S)$ and $\Psi = k\log\sum\abs{T_i}^2 + B(x)$ with continuious function $B$ near $S$. For $c(t) \in C^{\infty} ((-\infty,0), \mathbb{R}_{>0})$, we put $C(t) = c(t) e^t$, $D(t) = \int_t ^0 c(u)e^u du$ and $S(t) = \int_t ^{0} D(u)
du /D(t)$. If it suffices that $$(C,D,S) \in \cG_{(-\infty,0)}\;and$$ $$\lim_{t\rightarrow -\infty} D(t) <+\infty,$$ 
we have following Guan-Zhou's sharp $L^2$-extension theorem \cite{GZ15}.
\begin{thmC*}[Theorem 2.2 in \cite{GZ15}]\label{sharp L2 extension}
For any plurisubharmonic function $\phi$ on $X$ and holomorphic section $f \in \Gamma(S,K_X \vert_S)$, satisfying $$ \int_S  \imaginary^{(n-k)^2}\frac{f\wedge \bar{f}}{dT\wedge d\bar{T}} e^{-\phi-B} < +\infty,$$ there exists a holomorphic $n$--form $F$ on $X$ satisfying $F\vert_S = f$ and $$\int_X c(\Psi)  \imaginary^{n^2} F\wedge\bar{F} e^{-\phi} \le \frac{\pi^k}{k!} \left(\int_{-\infty}^0 c(t)e^t dt\right) \int_S  \imaginary^{(n-k)^2}\frac{f\wedge \bar{f}}{dT\wedge d\bar{T}} e^{-\phi-B}.$$
\end{thmC*}
In \cite{DNWZ23}, it is shown that the sharp $L^2$-extension property induces plurisubharmonicity. Observing the proof of Theorem C allows us to formulate an another $L^2$-division property that characterizes plurisubharmonic functions. We shall provide the details. We denote domain $\{z = (z_1,\dots,z_n) \in \mathbb{C}^n |\sum a_{i\bar{j}}z_i\overline{z_j} < 1\}$ by $E(A)$, where $A = \{a_{i\bar{j}}\}$ is a strict positive definite hermitian form. For some square root $L$ of $A$ (i,e $A = LL^*$), we put $w = (z_1,\dots,z_n)L$ then  $E(A) =\{\sum \abs{w_i}^2 < 1\}$.
\begin{thmD*}
    Let $\Omega$ be a domain in $\mathbb{C}^n$ and $\phi$ be a upper semicontinuious and $e^{-\phi}$ is locally integrable on $\Omega$.
    Assume that any $x \in \Omega$ and $E(A)$, inequality
        $$e^{-\phi(x)} \ge \limsup_{\epsilon \rightarrow +0}\frac{1}{\vol(\epsilon E(A))}\int_{x+ \epsilon E(A)} e^{-\phi}$$
    holds. If for any $x \in \Omega$, positive definite hermitian form $A$ and $\theta \in (0,\frac{\pi}{2})$ satisfying $x + E(A) \subset \Omega$, there exist holomorphic functions $F,G$ on $x+E(A)$ satisfies $F+G =1$ and $$\frac{1}{\cos^2 \theta}\int_{x+E(A)} \abs{G}^2 e^{-\phi} + \frac{1}{\sin^2\theta}\int_{x+E(A)} \abs{F}^2 \abs{w}^{-2n} e^{-\phi} \le \int_{x+E(A)} \frac{1}{(\sin^{2} \theta + \abs{w}^{2n}\cos^2{\theta})^2} e^{-\phi},$$
    then $\phi$ is plurisubharmonic.
\end{thmD*}
\begin{rem}
    In the above theorem, regularity property $$e^{-\phi(x)} \ge \limsup_{\epsilon \rightarrow +0}\frac{1}{\vol(x+\epsilon E(A))}\int_{x+ \epsilon E(A)} e^{-\phi}$$
    is necessary. As a counter example, the function $\phi$ which take a value $1$ at some point $x$, $0$ otherwise has the above sharp $L^2$-division property , but it is not plurisubharmonic.  
\end{rem}
\begin{section}{Positivity of chern curvature forms}
    In this section, we recall some basic results of positivity of chern curvature forms holomorphic vector bundles. Details can be found in \cite{Demailly lec}.  Let $(X,\omega)$ be a n-dimensional $\rm{\kahler}$ manifold. For a holomorphic hermitian line bundle $(L,e^{-\phi})$ on $X$, we say $(L,e^{-\phi})$ is a positive line bundle if its chern curvature $ \imaginary \Theta_L = \dd \phi$ is positive definite. The positivity of curvature is important because it ensures some vanishing theorems based on harmonic integral theory or $L^2$ theory of $\rm{\hormander}$. The proofs of these theorems fundamentally rely on the positive definiteness of twisted operator $[ \imaginary \Theta_L,\Lambda_{\omega}]$ on $\Wedge^{(n,q)} T_X$ for $q >0$, where $\Lambda_{\omega} $ is adjoint of $\omega$. To obtain similar vanishing theorems for higher rank hermite vector bundles $(E,h_E)$, the positivity of the chern curvature $ \imaginary \Theta_{h_E}$ is also required. However, in the case of vector bundles of rank two or higher, there exist essentially different definitions of positivity, and they produces different results. Let us first recall these definitions and review the corresponding results.

    \begin{dfn}\label{m-positivity}
        Let $V$, $W$ be complex vector spaces, and $B$ and $B'$ be hermitian forms on $V \otimes W$.
        \begin{enumerate}
            \item A tensor $u \in V \otimes W$ is said to be of rank m if $u$ can be written 
            \begin{equation}
                u = \sum_{i = 1}^{m} v_i \otimes w_i, \;\; v_i \in V \; w_i \in W.
            \end{equation}
            \item $B$ is said to be $m$-(semi) positive if for any $u \in V \otimes W$ of rank $m$, $B(u,u) > 0$ $(B(u,u) \ge 0)$. In this case we write
            \begin{equation}
                B >_m 0 \;\; (B \ge_m 0).
            \end{equation}
            
            \item If $B-B' >_* 0 \;(B-B'\ge_* 0)$, we write 
            \begin{equation}
                B >_* B' \;\; (B \ge_* B')
            \end{equation}
        \end{enumerate}
        \end{dfn}
        \begin{dfn}
            If for the chern curvature $\imaginary\Theta_{h_E}$, the hermitian form $\imaginary h_E(\Theta_{h_E}*,*)$ is (strictly) $m$-positive on $E\otimes T_X$, we write $\imaginary\Theta_{h_E} \ge_{m} 0\;$ $(\imaginary\Theta_{h_E} >_{m} 0)$. Especially, if $\imaginary\Theta_{h_E}>_{1} 0$, we say $E$ is Griffith positive. If $\imaginary\Theta_{h_E}>_{\min\{r_E,n\}} 0$, then we say $E$ is Nakano positive. 
        \end{dfn}
    For example, the hermitian vector bundle $(T_{\mathbb{P}^n},\omega_{FS})$ is not Nakano positive bundle but Griffith positive bundle for $n\ge 2$, where $\omega_{FS}$ is Fubini-Study metric.
    
    The key point is that if $\imaginary\Theta_{h_E}$ is $m$-positive, then $h_E([\imaginary\Theta,\Lambda_{\omega}]*,*)$ is positive definite on $E\otimes \Wedge^{(n,q)} T_X$ for $$\min(n,r_E)-m+1\le q \le \min(n,r_E).$$ Thus, if $m$-positivity holds for larger $m$, it follows that the vanishing of lower-order cohomology is established. In particular, Nakano positive is an important concept as it leads to the vanishing theorem of the first cohomology and directly guarantees the solvability of complex analysis equations, in a manner similar to the Oka-Cartan theorem.

    
\par
    Next, we recall the fundamental formulas or inequality concering the curvature of subbundles and quotient bundles of a hermitian vector bundle $E$.  Let $(E,h_E)$ be a $r_E$ ranked holomorphic hermitian vector bundle on n-dimensional complex manifold $X$, $Q$ be a $r_Q$ ranked holomorphic vector bundle, and $g\colon E \rightarrow Q$ be a surjective holomorphic morphism. Put $i\colon S \hookrightarrow E$ be a subbundle defined by $S = \ker{g}$, $h_S$ and $h_Q$ be hermitian metric induced by $i$, $g$.\\\\

    Let us recall the following Gauss-Codazzi type direct sum decomposition formula.
    \begin{lem}\label{basic curvature}
        Let $D_E$, $D_S$, $D_Q$, be chern connections and $\Theta_E$, $\Theta_S$, $\Theta_Q$ be chern curvatures. In the $C^{\infty}$ splitting $E \simeq S\oplus Q$, $D_E$ and $\Theta_E$ admits a matrix decomposition 
        \begin{gather}
            D_E = \begin{pmatrix}
                D_S & -\beta^*\\
                \beta & D_Q
            \end{pmatrix}\\       
            \Theta_E =\begin{pmatrix}
                \Theta_S - \beta^* \wedge \beta & A\\
                A^* & \Theta_Q - \beta \wedge \beta^*
            \end{pmatrix}
        \end{gather}
        where $\beta^* = \dbar g^* - g^* \dbar \in C^{\infty} (Hom(Q,S)\otimes \Wedge^{0,1} TX^{\lor})$, $\beta \in C^{\infty} (Hom(S,Q)\otimes \Wedge^{1,0} TX^{\lor})$ is adjoint of $\beta$, and $A^* = \dbar \beta$, $A$ is adjoint of $A^*$.
    \end{lem}
    
    The following is an important inequality for compareing the curvature of a subbundle with that of a quotient bundle.
    \begin{lem}[Proposition (11.16) in \cite{Demailly lec}]\label{skoda inequality}
        For integer $0\le q \le \min{(n,r_E - r_Q)}$, inequality
        \begin{gather}
            q\Tr\beta\wedge\beta^* \ge_q -\beta^*\wedge\beta.
        \end{gather}
        holds.
    \end{lem}
    We use following inequality to prove Theorem A.
    \begin{prop}
    If $ \imaginary \Theta_E \ge_q \Gamma \otimes Id_E$ for some real $(1,1)$ form $\Gamma$ and $\imaginary\Theta_E \ge_q 0$, then inequality
    \begin{equation}
         \imaginary \Theta_S + q \imaginary \Theta_{\det{Q}} \otimes Id_S \ge_{q} (q r_Q +1)\Gamma \otimes Id_S
    \end{equation}
    holds for $q = min (r_E - r_Q, n-k)$.
    \end{prop}
    \begin{proof}
        Since Lemma $\ref{basic curvature}$, we have
        \begin{align}
            &\imaginary \Theta_S = \imaginary \Theta_E|_S  + \imaginary \beta^* \wedge \beta
            \ge_{q} \Gamma \otimes Id_S +  \imaginary\beta^* \wedge \beta,\\
            &\imaginary\Theta_Q = \imaginary\Theta_E |_Q + \imaginary\beta \wedge \beta^* 
            \ge_1 \Gamma \otimes Id_Q + \imaginary\beta \wedge \beta^* \;\;\; and\\           
            &\imaginary\Theta_{\det{Q}}  = \Tr \imaginary \Theta_Q \ge r_Q\Gamma + tr\imaginary\beta \wedge \beta^*.        
        \end{align}
        Since Lemma \ref{skoda inequality}, we get 
        \begin{equation}
        \imaginary \Theta_S + q\imaginary \Theta_{\det{Q}} \otimes Id_S \ge_{q} (q r_Q +1)\Gamma \otimes Id_S.  
        \end{equation}
    \end{proof}
\end{section}
\begin{section}{Refinement of H\"ormander's estimate}
    The following $\rm{\hormander}$'s $L^2$-existence theorem is powerful tool used extensively throughout complex geometry. 
    \begin{thm}[\cite{Hor65}]\label{Hormander}  Let $X$ be an n-dimensional complete $\kahler$ manifold equipped with a (not necessarily complete) $\kahler$ metric $\omega$ and $\phi$ be a $C^2$ class plurisubharmonic function. Let $(E,h_E)$ be a $r_E$ ranked hermitian vector bundle and $\Theta_E$ is chern curvature, and for integer $1\le k \le n$, we put $q = \min{(r_E ,n-k+1)}$. Assume that $\imaginary \Theta_E \ge_q 0$. Then for any $E$ valued $\dbar$ closed $(n,k)$ form $\beta$ such that
        \begin{equation}
            I = \limsup_{\epsilon \rightarrow +0} \int_{X} \ip{(B_{\epsilon} \Lambda_{\omega})^{-1} \beta}{\beta} e^{-\phi} dV_{\omega} < +\infty
        \end{equation}
    where $B_{\epsilon} = \wedge(\dd \phi + \epsilon \omega)$ and $\Lambda_\omega$ is adjoint of $L_\omega = \wedge \omega$,  there exists $E$ valued $(n,k-1)$ form $u$ such that 
    \begin{gather}
        \dbar u = \beta\\
        \int_{X} \abs{u}_{h_E,\omega}^2 dV_{\omega} \le I.
    \end{gather}
        
    \end{thm}
    Skoda showed Theorem \ref{Skoda division} by using the $\rm{\hormander}$ estimate and the Gauss-Codazzi  type curvature formula. Since we want to show ``sharp'' $L^2$-division theorem, the following is a natural question:
    \begin{question}
        Is the H\"ormander estimate sharp ?(i.e Does the estimate contain non-trivial equality case?)
    \end{question}
    In our observation, when $\phi$ is exhaustion, it may include cases where the non-trivial equality holds (see bellow example \ref{equality hold in Hormander}). But when $\phi$ is bounded above, the equality cannot hold (see bellow (2) in example \ref{remark on a priori estimate}).
    \begin{exam}\label{equality hold in Hormander}
        Let $\Omega$ be an open unit ball $\mathbb{B}$ in $\mathbb{C}$, $\omega$ be a $\kahler$ form  $\imaginary dz \wedge d\bar{z}$, and $\phi = 2\log{(\frac{1}{1-\abs{z}^2})}$, $f = \frac{d\bar{z}\wedge dz}{(1-\abs{z}^2)^2}$. Then the $1$-form $u = \frac{\bar{z}dz}{1-\abs{z}^2}$ is the $L^2 (\mathbb{B},e^{-\phi})$ minimum solution of $\dbar u = f$. Now we have
        \begin{gather}
            \dd\phi = 2\frac{\imaginary dz\wedge d\bar{z}}{(1-\abs{z}^2)^2},\\
            \int_{\mathbb{B}} \abs{u}^2 e^{-\phi} \omega = \int_{\mathbb{B}} \abs{z}^2 \omega = \frac{\pi}{2}\;\;\;and\\
            \int_{\mathbb{B}} \ip{(\dd\phi \Lambda_{\omega})^{-1} f}{f}e^{-\phi} \omega = \int_{\mathbb{B}} \frac{1}{2} \omega  = \frac{\pi}{2}.
        \end{gather}
        Thus we gobtain following equality
        \begin{gather}
            \min_{\dbar u = f} \int_{\mathbb{B}} \abs{u}^2 e^{-\phi} \omega = \int_{\mathbb{B}} \ip{(\dd\phi \Lambda_{\omega})^{-1} f}{f}e^{-\phi} \omega .
        \end{gather}
    \end{exam}
    The following $L^2$-existence theorem allows us to obtain sharp $L^2$-estimate even when $\phi$ is bounded above (see (1) in example \ref{remark on a priori estimate} bellow).
    \begin{thm}\label{a Priori}
        Let $X$ be a n-dimensional complete $\kahler$ manifold equipped with a (not necessarily complete) $\kahler$ metric $\omega$ and let $(E,h_E)$ be a $r_E$ ranked hermitian vector bundle and $\Theta_E$ be a chern curvature. Let $\phi$ be a $C^2$ class function on $X$ and $(C,D,S) \in \cG_{\phi}$. Assume that $B = \imaginary S(\phi) \Theta_E + \dd \phi \otimes Id_E \ge_q 0$ for $k = \min(n-k+1,r_E)$. Then for any $E$ valued $\dbar$ closed $(n,n - k +1)$ form $\beta$ such that
        \begin{equation}
            I = \limsup_{\epsilon \rightarrow +0} \int_{X} D(\phi) \ip{(B_{\epsilon} \Lambda_{\omega})^{-1} \beta}{\beta} dV_{\omega} < +\infty
        \end{equation}
    where $B_{\epsilon} = B + \epsilon \omega \otimes Id_E$ and $\Lambda_\omega$ is adjoint of $L_\omega = \wedge \omega$,  there exists $E$ valued (n,0) form $u$ such that 
    \begin{gather}
        \dbar u = \beta\\
        \int_{X} C(\phi) \abs{u}_E ^2  dV_{\omega} \le I.
    \end{gather}
    Therefore this estimate is sharp (see $\rm{(1)}$ in Example \ref{remark on a priori estimate}).\\
    \end{thm}
    This $\dbar$-estimate technique appears essentially in \cite{GZ12},\cite{GZ15},\cite{GZ17} (or same O.D.E problem in proof of Theorem \ref{a Priori} appears in \cite{blocki suita}) for proving the sharp $L^2$-extension theorem. We can obtain Theorem \ref{a Priori} by appropriately choosing the tuple $(h_E,\eta,\lambda)$ in following Ohsawa-Takegoshi's a priori estimate, which is also used to prove original Ohsawa-Takegoshi extension theorem.
    \begin{lem}[\cite{OT87}]\label{OT a priori} Let $X$ be an n-dimensional complete $\kahler$ manifold equipped with a (not necessarily complete) $\kahler$ metric $\omega$ and $(E,h_E)$ be a $r_E$ ranked hermitian vector bundle and $\Theta_E$ is chern curvature.  For integer $1\le k \le n$, we put $q = \min{(r_E ,n-k+1)}$. Let $\eta$ be a $C^2$ class positive function and $\lambda$ be a continuous positive function. Assume that $B = \imaginary \eta \Theta_E - (\dd \eta \otimes Id_E +\partial \eta \wedge \dbar \eta/\lambda \otimes Id_E) \ge_{q} 0$. Then for any $E$ valued $\dbar$ closed $(n,k)$ form $\beta$ such that
        \begin{equation}
            I = \limsup_{\epsilon \rightarrow +0} \int_{X} \ip{(B_{\epsilon} \Lambda_{\omega})^{-1} \beta}{\beta} dV_{\omega} < +\infty
        \end{equation}
    where $B_{\epsilon} = B + \epsilon \omega \otimes Id_E$ and $\Lambda_\omega$ is adjoint of $L_\omega = \wedge \omega$,  there exists $E$ valued $(n,k-1)$ form $u$ such that 
    \begin{gather}
        \dbar u = \beta\\
        \int_{X} \frac{\abs{u}_E ^2} {(\eta + \lambda)}   dV_{\omega} \le I.
    \end{gather}
        
    \end{lem}
\begin{proof}[(proof of Theorem \ref{a Priori})]
    Let $T$ be a $C^2$ class function and $S$ be a $C^2$ class positive function satisfying $S\ddot{T} - \ddot{S} >0$. Put $\eta = S(\phi)$, $\lambda = \dot{S}^2(\Phi) /(S(\Phi)\ddot{T}(\Phi) - \ddot{S}(\Phi))$, replace $h_E$ to $\tilde{h}_E = e^{-T(\phi)} h_E$. Now, the twisted operator $B$ in Lemma \ref{OT a priori} is given by the expression
    \begin{equation}
        \begin{split}
        B &= \imaginary \eta \Theta_{E,\tilde{h}_E} - (\dd \eta \otimes Id_E + \imaginary \partial \eta \wedge \dbar \eta /\lambda \otimes Id_E)\\
        &= S \imaginary \Theta_{E,h_E} + (S\dot{T} - \dot{S}) \dd \phi \otimes Id_E + (S\ddot{T} - \ddot{S} - \dot{S}^2 /\lambda) \partial \phi \wedge \dbar \phi \otimes Id_E\\
        &= S \imaginary \Theta_{E,h_E} + (S\dot{T} - \dot{S}) \dd \phi \otimes Id_E.
    \end{split}
    \end{equation}
    Accordingly, we consider solving the following ordinary differential equation.
    \begin{align}
        &S\dot{T} - \dot{S} = 1 \label{ODE1}  \\
        &e^{-T}/(S + \lambda) = C \label{ODE2} \\
        &S\ddot{T} - \ddot{S} >0. \label{ODE3}
    \end{align}    
     Differentiate both sides of (\ref{ODE1}) respect to $t$, and we get 
     \begin{equation} 
        \label{ODE1'} S\ddot{T} - \ddot{S} = -\dot{S}\dot{T},
    \end{equation}
     and we also get
        \begin{equation}\label{b}
           S + \lambda = S + \frac{\dot{S}^2}{S\ddot{T} - \ddot{S}}
            = S - \frac{\dot{S}}{\dot{T}}
            = \frac{S\dot{T} - \dot{S}}{\dot{T}}
            =\frac{1}{\dot{T}}.
        \end{equation} 
    By substitusing (\ref{b}) into (\ref{ODE2}),  we have 
    \begin{equation}
    \dot{T}e^{-T} = C.
    \end{equation}
    As a solution to this ordinary differential equation, we obtain 
    \begin{equation}
        \label{ODE1 int}T= -\log{D} ,  
    \end{equation}
    where $D$ is a some positive primitive function of $-C$. Furthermore, it follows from (\ref{ODE1}) that
    \begin{equation}
    S = -e^T \int e^{-T},
    \end{equation} 
    this is the  solution of $\frac{d}{dt} (SD) = -D$. Finally, we examine the conditions under which $S$ satisfies (\ref{ODE3}).
    From the monotonic increase of $T$ and (\ref{b}), it follows that (\ref{ODE3})  is equivalent to monotonic decrease of $S$.
\end{proof}
\begin{exam}\label{remark on a priori estimate}
    \begin{enumerate}
    \item We can check sharpness even if $\phi$ is bounded. Let $X = \mathbb{B}$ be a open unit ball in $\mathbb{C}$ equipped with $\kahler$ metric $\omega =\imaginary dz\wedge d\bar{z}$ . Put $\beta_0 = d\bar{z} \wedge dz$ then $L^2 (\mathbb{B},d\lambda)$ bounded solution of $\dbar u = \beta_0$ is denoted as $u = (\bar{z} + f) dz$, where $d\lambda$ is the lebesgue measure and $f$ is a some $L^2 (\mathbb{B},d\lambda)$ bounded holomorphic function. We know that $f$ and $\bar{z}$  are orthogonal. Thus $L^2 (\mathbb{B},d\lambda)$ minimum solution of $\dbar u = \beta_0$ is the $\bar{z} dz$. Take $\phi = \abs{z}^2$ , and  $(C,D,S) \in \cG_{\phi}$ as 
\begin{equation}
    C(t) = 1,\; D(t) = 1-t,\; S = (1-t)/2.
\end{equation}
Then we have
\begin{equation}
    \int_{\mathbb{B}} C(\abs{z}^2) \abs{\bar{z}dz}^2 d\lambda = \int_{\mathbb{B}} \abs{\bar{z}}^2 d\lambda,
\end{equation}
and 
\begin{equation}
    \int_{\mathbb{B}} D(\abs{z}^2) \ip{(\dd \abs{z}^2 \Lambda_{\omega})^{-1} d\bar{z}\wedge dz}{d\bar{z}\wedge dz} d\lambda = \int_{\mathbb{B}} (1-\abs{z}^2) d\lambda.
\end{equation}
From $\int_{\mathbb{B}} d\lambda = \pi$, and $\int_{\mathbb{B}} \abs{z}^2
d\lambda = \frac{\pi}{2}$, we obtain the following
\begin{equation}
    \min_{\dbar u = \beta_0}\int_{\mathbb{B}} C(\abs{z}^2) \abs{u}^2 d\lambda = \int_{\mathbb{B}} D(\abs{z}^2) \ip{(\dd \abs{z}^2 \Lambda_{\omega})^{-1} \beta_0}{\beta_0} d\lambda.
\end{equation}

\item Take $(C,D,S) \in \cG_{\phi}$ as $C(t) = e^{-t}$, $D(t) =e^{-t} -e^{-M}$, $S(t) = (e^{-t} + (M-t -1)e^{-M})/(e^{-t} -e^{-M})$ where $M = \sup\phi$, we have
\begin{gather}
    \int_{\Omega} \abs{u}^2 e^{-\phi} \le \int_{\Omega} \ip{(\dd\phi \Lambda_{\omega})^{-1}f}{f} (e^{-\phi} -e^{-M}) dV_{\omega}.
\end{gather}
This is strictly stronger than Theorem \ref{Hormander}.
\end{enumerate}
\end{exam}
\end{section}

\begin{section}{Sharp $L^2$-estimate of Skoda-type division theorem}
    \subsection{Proof of Theorem A}
    \begin{proof}
        
        The proof basically follows the method of Skoda, but the $L^2$ estimate is replaced by Theorem \ref{a Priori}.

        (Step 1)
        First, we assume that $g$ is surjective everywhere. If we get a solution of $\dbar u = \dbar g^* f$ for some section $u$ on $S = \ker{g}$ with $L^2$ estimate
        \begin{gather}\label{GZ}
            \int_X C(\Phi) \abs{u}^2 e^{-q\Phi} dV_{\omega} \le \int_X qD(\Phi) \abs{g^*f}^2 e^{-q\Phi} dV_{\omega},
        \end{gather}
        then $F = g^* f - u$ is $\dbar$ closed and we get following $L^2$ estimate
        \begin{gather}
            \int_X C(\Phi) \abs{F}^2 e^{-q\Phi} dV_{\omega} \le \int_X (C(\Phi) + qD(\Phi)) \abs{g^*f}^2 e^{-q\Phi} dV_{\omega},
        \end{gather}
        since $S$ and $Q$ are orthogonal.\\
        Let $h_S$ and $h_Q$ are induced metric and $\Theta_{\det{Q}}$ is the chern curvature of $\det{h_Q}$. Put $\tilde{h}_S = e^{-q \Phi} h_S$. By the assumption $\{S(\Phi)(q r_Q +1) + r_Q \}\imaginary \Theta_E - (S(\Phi)q+1) \dd \phi \ge 0$, we have
        \begin{align}
            \imaginary \Theta_{S , \tilde{h}_S} &= \imaginary \Theta_S + q \dd \Phi \\
            &= \imaginary\Theta_S + q \imaginary \Theta_{\det{Q}} - q \dd \phi\\
            &\ge_q \left((qr_Q + 1)\frac{S(\Phi)q+1}{S(\Phi)(q r_Q +1) + r_Q} - q\right)\dd \phi\\
            &= \frac{1}{S(\Phi)(q r_Q +1) + r_Q} \dd\phi.
        \end{align}
        Thus we have
        \begin{align}
           B &= \imaginary S(\Phi)\Theta_{S , \tilde{h}_S} +\dd \Phi\\
           &= \imaginary S(\Phi)\Theta_{S , \tilde{h}_S} +\imaginary\Theta_{\det{Q}} -\dd\phi \\
           &\ge_q \imaginary S(\Phi)\Theta_{S , \tilde{h}_S} + \left( \frac{r_Q(S(\Phi)q + 1)}{S(\Phi)(q r_Q +1) + r_Q}-1\right)\dd\phi + \Tr \imaginary \beta \wedge \beta^* \otimes Id_S\\
           &\ge_q \imaginary S(\Phi)\Theta_{S , \tilde{h}_S} - \frac{S(\Phi)}{S(\Phi)(q r_Q +1) + r_Q}\dd\phi \otimes Id_S + \Tr \imaginary \beta \wedge \beta^* \otimes Id_S\\
           &\ge_q \Tr \imaginary \beta \wedge \beta^* \otimes Id_S.
        \end{align}
        From Lemma \ref{skoda inequality}, we also have
        \begin{align}
            \ip{(B\Lambda_{\omega})^{-1} \beta^* f} {\beta^* f} &\le
            \ip{(-\imaginary\Tr{\beta^* \wedge \beta}\Lambda_{\omega})^{-1} \beta^* f} {\beta^* f}\\
            &\le \ip{(-1/q \imaginary\beta^* \wedge \beta \Lambda_{\omega})^{-1} \beta^* f}{\beta^* f}\\
            &\le q\ip{g^*f}{g^* f}.
        \end{align}
    Since Theorem \ref{a Priori}, we obtain some solution of $\dbar u = \dbar g^* f$ with estimate (\ref{GZ}).\\
    \quad\par
    (Step 2)
    Even when $g$ is not surjective at some point, we can discuss same as (step1) through basic $L^2$ argument. We put $Z$ be the set $\wedge^{r_Q} g ^{-1} (0)$, and $\psi$ be a smooth exhaustion plurisubharmonic function on $X$. Then $X_j = \{\psi < j\} \setminus Z\;\; (j\in\N)$ is complete $\rm{\kahler}$ manifolds (see \cite{Demailly lec}). Thus by conducting the same computation as (Step 1) once again, we get solution $\dbar u_j = \beta^* f$ on $X_j$ with $L^2$-estimate
    \begin{align}
        \int_{X_j} C(\Phi) \abs{u_j}^2 e^{-q\Phi} dV_{\omega} \le \int_{X_j} qD(\Phi) \abs{g^*f}^2 e^{-q\Phi} dV_{\omega} \le \int_{X\setminus Z} qD(\Phi) \abs{g^*f}^2 e^{-q\Phi} dV_{\omega}.
    \end{align}
    Since weakly compactness of $L^2(X) \cap \ker \dbar$, we get global section $\lim_{j \rightarrow +\infty} u_j = u$ on $X\setminus Z$ such that
    \begin{gather}
        \dbar u = \beta^* f\;\;\;and\\
        \int_{X\setminus Z} C(\Phi) \abs{u}^2 e^{-q\Phi} dV_{\omega} \le \int_{X\setminus Z} qD(\Phi) \abs{g^*f}^2 e^{-q\Phi} dV_{\omega}.
    \end{gather} 
    Since $Z$ is a thin analytic set, the $\dbar$-closed form $u$ can be extended beyond $Z$. Thus we can get a $\dbar$-closed solution $F=g^*f - u$ on X with desired $L^2$ estimate.
    This means that proof of Theorem {\rm A} complete.
\end{proof}
\subsection{Examples}

First, the method for finding triples of $(C,D,S) \in \cG$ will be explained.
\begin{lem}\label{methods for finding}
Assume that  $\Phi < A$ for real number $A \in \mathbb{R}$ and $C$ is a positive $C^1$ class function on $(-\infty,A]$. For $\alpha \ge 0$, we put
\begin{gather}
    D(t) = \int_{t}^{A} C(s) ds + \alpha C(A),\\
    E(t) = \int_{t}^{A} D(s) ds + \alpha^2 C(A) \;\; and\\
    S(t) = E(t)/D(t).
\end{gather}
If inequality
\begin{gather}\\\label{suffices}
    \frac{d}{dt} C(t) > -\frac{C(t)D(t)}{E(t)}
\end{gather}
hold, then $(C,D,S) \in \cG_\Phi$ and $S(t) \ge \alpha$ for all $t \in (-\infty,A]$.
\end{lem}
\begin{proof}
    It is clearly that $(C,D,S)$ satisfies O.D.E (\ref{ode}) and $S(A) = \alpha$. Thus it suffices to show only $$\frac{d}{dt} S(t) = \frac{E(t)C(t) - D(t)^2}{D(t)^2} < 0.$$
Now we have $\frac{d}{dt} S(A) =0$ and \ref{suffices} implies that $$\frac{d}{dt}( E(t)C(t) - D(t)^2 )<0.$$ Thus we have proved lemma.
\end{proof}

\subsubsection[Psuedoconvex domain case]{Psuedoconvex domain case}\label{Psuedoconvex domain case}
\begin{exam}Let $\Omega$ be a psuedoconvex domain in $\mathbb{C}^n$, and let $\phi$ be a plurisubharmonic function on $\Omega$. Let $g = (g_1,\dots g_r) $ be a $r$ tuple of holomorphic functions on $\Omega$ such that $\abs{g}^2 < 1$, and put $q = \min{(n,r-1)}$.
    \begin{enumerate}  
        \item Take $(C,D,S) \in \cG_{(-\infty,0)}$ as 
            $C(t) = e^{-\delta t}$,
            $D(t) = \frac{(e^{-\delta t} - 1)}{\delta}$ and
            $S(t) = \frac{(e^{-\delta t} +\delta t - 1)}{\delta(e^{-\delta t} - 1)}$
        for $\delta >0$, then we get Corollary \ref{refinement of skoda}. Especially, this $(C,D,S)$ has sharp condition for $0<\delta<1$.
        \item Take $(C,D,S) \in \cG_{(-\infty,0)}$ as $C(t) = 1$, $D(t) = -t$, $S(t) = -t/2$. Then for any holomorphic function $f$ on $\Omega$ such that
            \begin{equation}
                \int_{\Omega} (1-q\log\abs{g}^2)\abs{f}^2 \abs{g}^{-2(q+1)} e^{-\phi} < +\infty,
            \end{equation}
        there exists $r$ tuple of holomorphic functions $F = (F_1,\dots ,F_r) $ such that
        \begin{gather}
            \sum g_i F_i = f\\
            \int_{\Omega} \abs{F}^2 \abs{g}^{-2q} e^{-\phi} \leq \int_{\Omega} (1-q\log\abs{g}^2)\abs{f}^2 \abs{g}^{-2(q+1)} e^{-\phi}.
        \end{gather} 
        \item Take $(C,D,S) \in \cG_{(-\infty,0)}$ as 
            $C(t) = e^{qt}$,
            $D(t) = \frac{(1 -e^{qt})}{q}$ and
            $S(t) = \frac{(e^{qt} -1 - qt)}{q(1 - e^{qt})}$.
        Then for any holomorphic function $f$ on $\Omega$ such that
            \begin{equation}
                \int_{\Omega} \abs{f}^2 \abs{g}^{-2(q+1)} e^{-\phi} < +\infty,
            \end{equation}
        there exists $r$ tuple of holomorphic functions $F = (F_1,\dots ,F_r) $ such that
        \begin{gather}
            \sum g_i F_i = f\\
            \int_{\Omega} \abs{F}^2 e^{-\phi} \leq \int_{\Omega}\abs{f}^2 \abs{g}^{-2(q+1)} e^{-\phi}.
        \end{gather}    
    \end{enumerate}
    
\end{exam}
\subsubsection[Division problem for holomorphic sections on vector bundles]{Division problem for holomorphic sections on vector bundles}
\begin{exam}
    Let $(X,\omega)$ be a n-dimensional compact $\kahler$ manifold and 
     let $Q$ be a $r_q$ ranked holomorphic vector bundle, and $e^{-\phi}$ is a smooth metric of $\det{Q}$ and let $(L,e^{-\varphi})$ be a hermitian line bundle. Let $\sigma_1,\dots \sigma_r$ be a non zero holomorphic section of $Q$ and let $V$ be a complex vector space $\langle \sigma_1 \dots \sigma_r \rangle$, and $E$ be a trivial vector bundle $X\times V$ equipped with hermitian metric $h_E$ such that $\sigma_i$ be orthogonal basis as fibrewise.\\

     Assume natural morphism $g \colon E \rightarrow Q$ is generically surjective, and a
    curvature inequality
    \begin{gather}
        \dd \varphi - (q+\epsilon) \dd \phi \ge 0
    \end{gather}
    holds for $\epsilon > 0$ and $q = \min(n, r - r_Q)$. If $$S \ge 1/\epsilon$$ for $(C,D,S) \in \cG_{\Phi}$,  then {\rm(\ref{condition 1})},{\rm(\ref{condition 2})} in {\rm Theorem A} holds for $g \colon E\otimes L \rightarrow Q\otimes L$. In addition, we can assume that $\Phi < 0$, by replacing $\Phi$ with $\Phi- \sup_X \Phi$. \\
    \begin{enumerate}
        \item Take 
        \begin{gather}
           C(t) = 1,\\
           D(t) = -t + \frac{1}{\epsilon} \;\;\; and\\
           E(t) = \frac{t^2}{2} -\frac{t}{\epsilon} + \frac{1}{\epsilon^2}.
        \end{gather}
        Then $(C,D,S = E/D) \in \cG$ and we have $S\ge \frac{1}{\epsilon}$ from {\rm Lemma \ref{methods for finding}}. 
        Thus for any holomorphic section $f$ of $Q\otimes L \otimes K_X$ such that 
        \begin{equation}
            \int_X \left(1 - q\Phi + \frac{q}{\epsilon}\right) \, \abs{g^*f}_{h_{E \otimes L,\omega}}^2 e^{-\varphi -q\Phi} dV_{\omega}  < \infty,
        \end{equation}
        there exists $r$ tuple of holomorphic section$(F_1 \dots F_r)$ of $L\otimes K_X$ such that
        \begin{gather}
            \sum \sigma_i F_i = f\\
            \sum \int_X F_i \wedge \bar{F_i} \,e^{-q\Phi - \varphi}  \leq \int_X \left(1 - q\Phi + \frac{q}{\epsilon}\right) \, \abs{g^*f}_{h_{E \otimes L,\omega}}^2 e^{-\varphi -q\Phi} dV_{\omega}. 
        \end{gather}
    \item For $0<\delta<\epsilon$, take
    \begin{gather}
        C(t) = e^{-\delta t}\\
        D(t) = \frac{1}{\delta} (e^{-\delta t} -1) + \frac{1}{\epsilon} \;\;\; and\\
        E(t) = \frac{1}{\delta^2} (e^{-\delta t} -1 + \delta t)+ \frac{1}{\epsilon^2 }- \frac{t}{\epsilon}.
    \end{gather}
    We can easily check that $E(t)$ is log concave and $S(0) = \frac{1}{\epsilon}$ so that $(C,D,S) \in \cG_{\Phi}$ and we get $S > \frac{1}{\epsilon}$.
    Thus for any holomorphic section $f$ of $Q\otimes L \otimes K_X$ such that 
        \begin{equation}
            \int_X \left(\left(1+\frac{q}{\delta}\right)e^{-\delta\Phi} +\frac{q}{\epsilon}-\frac{q}{\delta}\right)\abs{g^*f}_{h_{E \otimes L,\omega}}^2 e^{-\varphi -q\Phi} dV_{\omega}< \infty,
        \end{equation}
        there exists $r$ tuple of holomorphic section$(F_1 \dots F_r)$ of $L\otimes K_X$ such that
        \begin{gather}
            \sum \sigma_i F_i = f\\
            \sum \int_X F_i \wedge \bar{F_i} \,e^{-(q+\delta)\Phi - \varphi}  \leq \int_X \left(\left(1+\frac{q}{\delta}\right)e^{-\delta\Phi} +\frac{q}{\epsilon}-\frac{q}{\delta}\right)\abs{g^*f}_{h_{E \otimes L,\omega}}^2 e^{-\varphi -q\Phi} dV_{\omega}. 
        \end{gather}
        Letting $\delta \rightarrow \epsilon$, we get Theorem \ref{Skoda division}. And letting $\delta \rightarrow 0$, we get above example (1).
    \item Take
    \begin{gather}
        C(t) = e^{qt}\\
        D(t)  = \frac{(1-e^{qt})}{q} + \frac{1}{\epsilon}\;\;\; and\\
        E(t) =\int_{t}^{0} D(s) ds + \frac{1}{\epsilon^2}
    \end{gather}
    Then $(C,D,S = E/D) \in \cG$ and we have $S\ge \frac{1}{\epsilon}$ from Lemma \ref{methods for finding}. Thus for any holomorphic section $f$ of $Q\otimes L \otimes K_X$ such that 
    \begin{equation}
        \int_X \abs{g^*f}_{h_{E \otimes L,\omega}}^2 e^{-\varphi -q\Phi} dV_{\omega}  < \infty,
    \end{equation}
    there exists $r$ tuple of holomorphic section$(F_1 \dots F_r)$ of $L\otimes K_X$ such that
    \begin{gather}
        \sum \sigma_i F_i = f\\
        \sum \int_X F_i \wedge \bar{F_i} \,e^{- \varphi}  \leq \left(1+\frac{q}{\epsilon}\right)\int_X \abs{g^*f}_{h_{E \otimes L,\omega}}^2 e^{-\varphi -q\Phi} dV_{\omega}
    \end{gather}
    \end{enumerate}

\end{exam}
\end{section}

\begin{section}{Characterizing plurisubharmonicity via the sharp $L^2$-division property}\label{Characterizing plurisubharmonicity via the sharp L^2-division property}
    \subsection{Correspondences between the solvability of analytic equations and the convexities}
    In several complex variables, studying the relationship between convexities and solvability of analytic equations are fundamental. Kiyoshi Oka developed crucial techniques for patching sections of analytic coherent sheaves in solving the Cousin problem or Levi's problem. Henri Cartan later reformulating these ideas in the modern sheaf theory and cohomology theory: $H^q (\Omega,\cF) =0$ for $q>0$ and analytic coherent sheaf on psuedoconvex domain $\Omega$.
    This implies that there are no obstructions to solve many other holomorphic equations or $\dbar$-equations on psuedoconvex domain. Later, $\rm{\hormander}$ showed the converse by using $L^2$-method and stein theory. Thus it can be seen that the geometric convexity of the domain is equivalent to the solvability of the holomorphic equations.
    In recent years, further developments have revealed that $L^2$ solvability is equivalent to the plurisubharmonicity of functions (\cite{DNWZ23},\cite{DWZ18}).
    We recall the following $L^2$-property appeared in \cite{DNWZ23}.
    \begin{dfn}\label{L2 properties}
        \begin{enumerate}
            Let $\Omega$ be a domain in $\mathbb{C}^n$ and $\phi$ be a upper semicontinuious function.
            \item
            We say that $\phi$ has $L^2$-estimate property if following hold: Assume that for any smooth $\dbar$-closed $(0,1)$-form $\beta$ on $\Omega$ and smooth strictly plurisubharmonic $\psi$,  there exist smooth function $u$ such that $\dbar u = \beta$ and $$\int_{\Omega} \abs{u}^2 e^{-\phi-\psi} \le \int_{\Omega} \Tr_{\dd\psi} (\bar{\beta} \wedge \beta)e^{-\phi-\psi},$$
            if the right hand integral convergents.
            \item We say that $\phi$ has sharp $L^2$-extension property if following hold: For any $z$ in $\Omega$ and holomorphic cylinder $P$ centerd at $z$, there exist a holomorphic function $f \in \cO(P)$ such that $f(z) = 1$ and $$ \int_P \abs{f}^2 e^{-\phi} d\lambda \le \vol{(P)}\, e^{-\phi(z)}. $$
            Here by a holomorphic cylinder we mean a domain of the form $A(P_{r,s}) + z$ for some $A \in U(n)$ and $r,s>0$, with $P_{r,s} = \{(z_1,\dots,z_n) : \abs{z_1} <r, \abs{z_2}^2 + \dots +\abs{z_n}^2 <s^2\}$.
        \end{enumerate}
    \end{dfn}
    We know that plurisubharmonic function has the $L^2$-estimate property and the sharp $L^2$-extension property (\cite{BL14},\cite{blocki suita},\cite{GZ12} and \cite{Hor65}). In \cite{DNWZ23}, they proved that only plurisubharmonic function has $L^2$-estimate property. In other word, the $L^2(\Omega,e^{-\phi})$  solvability and the plurisubharmonicity of $\phi$ is equivalent.
    Theorem B provide a reasonable $L^2$-property characterizing plurisubharmonicity via the division problem.
    \subsection{Proof of Theorem B}
    We denote domain $\{z = (z_1,\dots,z_n) \in \mathbb{C}^n |\sum a_{i\bar{j}}z_i\overline{z_j} < 1\}$ by $E(A)$, where $A = \{a_{i\bar{j}}\}$ is a strict positive definite hermitian form. For some square root $L$ of $A$ (i,e $A = LL^*$), we put $w = (z_1,\dots,z_n)L$ then  $E(A) =\{\sum \abs{w_i}^2 < 1\}$.

   Let $\phi$ be a $C^2$ class function on a domain $\Omega$. Let us first observe the $L^2$-estimate for the division problem around a point $x$ where $d\phi(x) = 0$. In this case, we can show that the estimate of Theorem \ref{refinement of skoda} breaks down if the complex hessian $\dd\phi(x)$ has some negative eigenvalues.
   \begin{lem}\label{converce lemma 1}
    Let $(C,D,S) \in \cG_{(-\infty,0)}$ satisfies sharp condition, and $\phi$ be a $C^2$ class function defined around $0 \in \mathbb{C}^n$. Assume that
        $d\phi(0) =0$ and
        $\dd \phi(0)$ has negative eigenvalue.
    Then there exists sufficient large positive definite hermitian metric $A = \{a_{i\bar{j}}\}$ and holomorphic function $f$ and the tuple of holomorphic functions $g = (g_1,\dots,g_n)$ with $\abs{g}^2 <1$,
    the inequality
    \begin{equation}
        I_{A,\phi} \colon= \inf_{\substack{F=(F_1,\dots,F_n) \in \cO^{\oplus n} (E(A))\\\sum \alpha_i w_i F_i = f}}
        \frac{\int_{E(A)} C(\log\abs{g}^2) \abs{F}^2 \abs{g}^{-2(n-1)} e^{-\phi} d{\lambda}}
        {\int_{E(A)} (C(\log\abs{g}^2) +qD(\log\abs{g}^2))\abs{f}^2 \abs{g}^{-2n} e^{-\phi}d{\lambda}} >1
    \end{equation}
    holds, where $\{\alpha_i\}_i$ are eigenvalues of $A$, and $\{w_i\}_i$ are corresponding unit eigenvectors.
   \end{lem}
   \begin{proof}
    We can assume  that $\phi(0) = 0$. And we can choice coordinate system ${(w_1,\dots ,w_n)}$ by unitary transformation such that 
    \begin{equation}
        \phi(w) = \sum a_i \abs{w_i}^2 + O(\abs{w}^3)
    \end{equation}
    where $a_i$ are real number and $a_1<0$. Put $\phi_2 = \sum a_i \abs{w_i}^2$. Take sufficiently small positive number $\delta$, we change coordinate $(\tilde{w_1},\dots,\tilde{w_n}) = (\delta w_1,\dots,w_n)$ and we get
    \begin{equation}
        \int_{\abs{\tilde{w}}^2 = 1} \phi_2  d\mu < 0.
    \end{equation}
    where $d\mu$ is the H\"oussedolf measure of set ${\abs{\tilde{w}}^2 = 1}$.
    Put $z = (z_1,\dots,z_n) = \epsilon\tilde{w}$, and $A_\epsilon$ be hermitian metric such that $z$ be orthogonal coordinate system.
    Then we can denote
    \begin{equation}
        e^{-\phi} = 1 -\epsilon^2 \tilde{\phi}_2 (z) + O(\epsilon^3)
    \end{equation}
    where $\tilde{\phi}_2(z)$ is hermitian 2-form such that
    \begin{equation}
        W = \int_{\partial E(A_1)} \tilde{\phi}_2(z) d\mu < 0.
    \end{equation}
    For $f = \sum z_i$ and $g=(z_1,\dots,z_n)$, let $F = (F_i) \in \cO^{\oplus n} (E_0(A_{\epsilon}))$ be arbitrary holomorphic solution of $\sum z_i F_i = \sum z_i$. Then we can denote 
    \begin{gather}
    F = (1,1,\dots,1) + G\;\;\; and\\
    \abs{F}^2 = n + {\abs{G}^2} + \sum \Re\,{G_i} ,
    \end{gather} 
    where $G_i$ are vanish at origin at least order 1, and $\sum z_i G_i = 0$.
    \\ Since $\Re\,{{G_i}}$ and $1 -\epsilon^2 \tilde{\phi}_2 (z)$ are orthogonal in $L^2 (E(A_{\epsilon}),C(\log\abs{g}^2)\abs{g}^{-2(n-1)})$. Then we get
    \begin{equation}
        (\vol(E(A_{\epsilon})))^{-1}\int_{E(A_{\epsilon})}C(\log\abs{g}^2)\abs{g}^{-2(n-1)} e^{-\phi} ({\abs{G}^2} + \sum \Re{G_i}) d\lambda \ge - O(\epsilon^3).
    \end{equation}
    Put $\sigma$ be the volume of surface of $2n$-dim unit ball, we get
    \begin{gather}
        (\vol(E(A_{\epsilon})))^{-1}\int_{E(A_{\epsilon})}C(\log\abs{g}^2)\abs{g}^{-2(n-1)} e^{-\phi} {\abs{F}^2} d\lambda \\
        \ge \int_{0}^{1} n(\sigma - \epsilon^2 r^2 W)\,r\,C(\log{r^2})dr -O(\epsilon^3)
    \end{gather}
    and
    \begin{gather}
        (\vol (E(A_{\epsilon})))^{-1}\int_{E(A_{\epsilon})}(C(\log\abs{g}^2)+(n-1)D(\log{\abs{g}^2}))\abs{g}^{-2n} e^{-\phi} {\abs{f}^2} d\lambda \\
        \le \int_{0}^{1} (\sigma - \epsilon^2 r^2 W)\,r\,(C(\log{r^2})+(n-1)D(\log{r^2}))dr +O(\epsilon^3).
    \end{gather}
    Since
    \begin{equation}
        \int_{0}^{1} rC(\log{r^2}) dr = \int_{0}^{1} rD(\log{r^2})
    \end{equation}
    and
    \begin{gather}
       \int_{0}^{1} r^3 D(\log{r^2}) dr
       =\left[ \frac{r^4}{4} D(\log{r^2})\right]_{0} ^{1} + \frac{1}{2}\int_{0}^{1} r^3C(\log{r^2}) dr \\
       = \frac{1}{2}\int_{0}^{1} r^3C(\log{r^2}) dr.
    \end{gather}
    We can conclude $I_{A_\epsilon,\phi}  =\frac{nM_1\sigma -M_2\epsilon^2 W - O(\epsilon^3)}{nM_1\sigma -M_3\epsilon^2 W + O(\epsilon^3)}>1$ for sufficiently small $\epsilon >0$, where $ M_1 = \int_{0}^{1}rC(\log{r^2})dr$ , $ M_2 = n\int_{0}^{1}r^3C(\log{r^2})dr$ , $ M_3 = \frac{n+1}{2}\int_{0}^{1}r^3C(\log{r^2})dr$.
   \end{proof}
   Next, we consider the case where $d\phi(x) \ne 0$. In this case, we can also show that the sharp estimate breaks down if the $\dd e^{\phi}(x)$ has some negative eigenvalues.
   \begin{lem}\label{converce lemma 2}
        Let $(C,D,S) \in \cG_{(-\infty,0)}$ satisfies sharp condition, and let $\phi$ be a $C^2$ class function defined around $0 \in \mathbb{C}^n$ $(n \ge 2)$. Assume that $\dd \phi(0) + \imaginary\partial\phi\wedge\dbar \phi(0)$ has negative eigenvalue. Then there exists psuedoconvex domain $0\in\tilde{\Omega}$, and non zero holomorphic function $f$ and $n$ tuple of holomorphic function $g= (g_1,\dots,g_n)$ such that $dg_1\wedge\dots \wedge g_n \ne 0$ and $\abs{g}^2 <1$ and
        \begin{equation}\label{converse ie 2}
            \inf_{\substack{F=(F_1,\dots,F_n) \in \cO^{\oplus n} (\tilde{\Omega})\\\sum g_i F_i = f}}
        \frac{\int_{\tilde{\Omega}} C(\log\abs{g}^2) \abs{F}^2 \abs{g}^{-2(n-1)} e^{-\phi} d{\lambda}}
        {\int_{\tilde{\Omega}} (C(\log\abs{g}^2) +(n-1)D(\log\abs{g}^2))\abs{f}^2 \abs{g}^{-2n} e^{-\phi}d{\lambda}} >1.
        \end{equation}
   \end{lem}
   \begin{proof}
    We can assume that $\phi(0) =0$. Since Lemma\ref{converce lemma 1}, if $d\phi(0) = 0$, then inequality \ref{converse ie 2} holds. Thus we assume that $d\phi(0) \ne 0$. We get taylor series
    \begin{equation}
        \phi(z) = \sum a_i \Re z_i + \sum b_i \Im z_1 + \phi_2 + O(\abs{z}^3)
    \end{equation}
    where $a_i$ and $b_i$ are some real number, and $\phi_2$ is a second order term. By transformation $(z_1,\dots, z_n) \rightarrow (e^{i\theta_1}z_1,\dots,e^{i\theta_n}z_n)$, we get
    \begin{equation}
        \phi(z) = \sum c_i \Re z_i + \tilde{\phi_2 } + O(\abs{z}^3)
    \end{equation}
    where $ c_i $ are some real numbers and $\tilde{\phi_2}$ is a second order term. By some unitary transformation, we geometry coordinate $(w_1,\dots,w_n)$ such that
    \begin{equation}
        \phi(w) = c \Re w_1 +\tilde{\phi_2}(w) + O(\abs{w}^3)
    \end{equation}
    where $c$ is a non zero real number. By local coordinate change $(u_1,\dots,u_n) = (w_1 - \frac{c}{4} w_1 ^2,\dots,w_n)$,  we get
    \begin{gather}
        e^{-\phi(w)} d\lambda(w) = e^{-\Phi(u)} d\lambda(u),\\
        \Phi(u) = \tilde{\phi_2}(u) + \frac{c^2}{4} \abs{u_1}^2 + O(\abs{u}^3).
    \end{gather}
    Thus we get matrix representation 
    \begin{equation}
        \dd\Phi(0) = \dd\phi(0) + \imaginary\partial\phi(0)\wedge\dbar\phi(0).
    \end{equation}
    and $\dd\Phi(0)$ has a some negative eigenvalue. Since Lemma \ref{converce lemma 1}, inequality (\ref{converse ie 2}) holds.
   \end{proof}
   We are now ready to prove the Theorem B.
   \begin{proof}[(proof of {\rm Theorem B})]
    \quad\par \quad\par
    (1) in Theorem B follows immediately from Lemma \ref{converce lemma 2}.\\

    If $\dd\phi$ has some negative eigenvalue at some point $x$, then for sufficiently small $\epsilon >0$, $\dd (\epsilon\phi)(x) + \imaginary \partial (\epsilon\phi)(x) \wedge \dbar (\epsilon\phi)(x)$ has a some negative eigenvalue. Thus (2) follows from (1)\\

    If $\dd\phi$ has some negative eigenvalue at some point $x$, then we can take affine function $\psi$ such that $d(\phi +\psi)(x) =0$. Thus (3) follows from Lemma \ref{converce lemma 1}.
   \end{proof}
\end{section}
\begin{section}{A proof of sharp $L^2$-extension theorem}\label{A proof of sharp L2 extension theorem}
    In this section, we provide a new proof of Guan--Zhou's sharp $L^2$-- extension theorem via the sharp $L^2$--division theorem. Again, the idea is essentially based on \cite{A25}.
    First, we explain the $L^2$--division theorem for multiplier ideal sheaves, which essentially contains Skoda's original division theorem.\\
    
    \begin{thm}\label{division theorem for multiplier ideal sheaf}
        Let $X$ be an n-dimensional stein manifold, $\psi,\phi_1,\dots,\phi_r$ be plurisubharmonic functions such that $\Phi = \log(\sum_{i=1}^r e^{\phi_i}) < 1$, and $(C,D,S) \in \cG_{(-\infty,0)}$. We put $q = \min(n,r-1)$. For any holomorphic n-form $f$ on $X$ such that $$\int_{X} (C(\Phi)+qD(\Phi)) \imaginary^{n^2} f \wedge\bar{f} e^{-\psi -(q+1)\Phi}< +\infty$$, there exists holomorphic n-forms $F_i$ such that $ \sum F_i = f$ and $$\sum_{i=1}^r \int_{X}  \imaginary^{n^2} C(\Phi) F_i \wedge \bar{F}_i e^{-\phi_i-q\Phi-\psi} \le \int_{X} (C(\Phi)+qD(\Phi)) \imaginary^{n^2} f \wedge\bar{f} e^{-\psi -(q+1)\Phi}.$$
        It implies $\cI((q+1)\log(\sum e^{\phi_i})) \subset \sum \cI(\phi_i)$. 
    \end{thm}
    \begin{proof}
    The proof is obtained by approximating the singular Hermitian metric $h = \diag(e^{-\phi_1},\dots,e^{-\phi_r})$ of the trivial bundle $E = X\times \mathbb{C}^r$ with smooth Hermitian metrics whose chern curvature is Nakano semi-positive.
    
    We can take a increasing sequence of relatively compact stein domains $\{X_m\}$ such that $\cup_m X_m = X$. Then we also take decreasing sequence of smooth plurisubharmonic functions $\{\phi_{i,m,k}\}$ on $X_m$ such that $\lim_{k\rightarrow +\infty} \phi_{i,m,k} = \phi_i$ on $X_m$. Then the smooth hermitian metric $h_{m,k} = \diag(e^{-\phi_{1,m,k}},\dots,e^{-\phi_{r,m,k}})$ of the  trivial vector bundle $E = X \times \mathbb{C}^r$ is Nakano semi-positive. Applying Theorem A to the data $\{X_m,h_{m,k}\}$ with $(C,D,S) \in \cG_{(-\infty,0)}$ and letting $k,m \rightarrow \infty$, we obtain assertion.
    \end{proof}
    For holomorphic functions $g_1,\dots,g_r$, we put $\phi_i = \log{\abs{g_i}^2}$, we recover the division problem for holomorphic functions.\\
    Nextly, to prove Theorem C, we present a following calculus lemma.
\begin{lem}\label{calculus lemma}
    Let X be a complex manifold and $\Omega$ be a relatively compact sub-domain. Let $S$ be a smooth closed submanifold in $X$ with $\codim S = k$. Assume that there exists holomorphic functions $T_1,\dots,T_k$ such that $$ S = \{T_1 = \dots = T_k =0\}\;\;\;and$$ $dT \ne 0$ on $S$ where $dT = dT_1 \wedge \dots \wedge dT_k$. For any continuous semipositive real $(n,n)$-form $dV$ and continuious function $B(x)$ on $X$, we have
    $$\lim_{t \rightarrow +\infty} \frac{1}{1+t}\int_{\Omega} \frac{dV}{(\frac{1}{1+t} + \frac{t}{1+t} e^{B(x)}\abs{T}^{2k})^2} = \frac{\pi^{k}}{k!}\int_{\Omega \cap S} \frac{dV e^{-B}}{\i^{k^2} dT \wedge d\bar{T}}.$$
    
\end{lem}
\begin{proof}
    Since $\Omega$ be a relatively compact, thus $dV$ is uniformly continuious and integrable on $\Omega$. For any $s \in \Omega \cap S$, there exists local cordinate $((\mathbb{B}_k (1) \times \mathbb{B}_{n-k} (1) ), (z_1,\dots,z_k,w_1,\dots,w_{n-k}))$ such that $s = 0$ and $z_i = T_i$ for $i = 1,\dots,k$. Thus it suffice to show in the case that $X = \mathbb{B}_k (1) \times \mathbb{B}_{n-k}(1)$, $S = \mathbb{B}_{n-k} (1)$ and $T_i = z_i$. For any $\delta >0$, 
    $$\lim_{t \rightarrow +\infty} \frac{1}{1+t}\int_{\mathbb{B}_k (1) \times \mathbb{B}_{n-k} (1) \backslash \mathbb{B}_k (\delta) \times \mathbb{B}_{n-k} (1) } \frac{dV}{(\frac{1}{1+t} + \frac{t}{1+t} e^{B(z,w)}\abs{z}^{2k})^{2}} = 0.$$
    For any $\epsilon>0$, we can take sufficiently small $\delta_{\epsilon}>0$ with $$dV(z,w) \ge (1+\epsilon)dV(0,w)\;\;\;and$$ $$B(z,w) \ge B(0,w) + \epsilon $$on $\mathbb{B}_k (\delta_{\epsilon}) \times \mathbb{B}_{n-k} (1)$. Thus by taking sufficiently small $\epsilon$ and Fubini's lemma, it also suffices to show that $$\lim_{t \rightarrow +\infty} \frac{1}{1+t}\int_{\mathbb{B}_k (1)} \frac{d\lambda}{(\frac{1}{1+t} + \frac{t}{1+t} e^{B}\abs{z}^{2k})^{2}} = \frac{\pi^{k} e^{-B}}{k!}$$ for a constant $B$. It follows that

    \begin{align}
    \frac{1}{1+t}\int_{\mathbb{B}_k (1)} \frac{d\lambda}{(\frac{1}{1+t} + \frac{t}{1+t}e^{B} \abs{z}^{2k})^2}&= 2k\times \frac{\pi^{k}}{k!}\frac{1}{1+t}\int_{0}^{1} \frac{r^{2k-1}dr}{(\frac{1}{1+t} + \frac{t}{1+t} e^{B}\abs{r}^{2k})^2}\\ 
    &= 2k\times \frac{\pi^{k}}{k!} (1+t) \int_{0}^{1} \frac{r^{2k-1}dr}{(1 + e^{B}tr^{2k})^2}\\
    &=  2k\times \frac{\pi^{k}}{k!} \frac{(1+t)e^{-B}}{t}\int_{0}^{e^{\frac{k}{2}}t^{\frac{1}{k}}} \frac{r^{2k-1}dr}{(1 + r^{2k})^2}\\
    &\rightarrow \frac{\pi^{k}e^{-B}}{k!}.
\end{align}

\end{proof}
\begin{proof}[Proof of Theorem C]
First, we prove theorem on relatively compact stein sub-domain $Y \subset X$. In addition, we assume that $\phi$ is continuious. 
For the positive real number $t$, we put tuple of plurisubharmonic functions $(\phi_1,\phi_2) = \left(\log(\frac{t}{1+t})+\Psi ,\log(\frac{1}{1+t})\right)$ and $\Psi_t = \log(\frac{t}{1+t}e^{\Psi} + \frac{1}{1+t})$. Since $X$ is stein, we can find a holomorphic $n$-form $\tilde{F}$ on X such that $\tilde{F}\vert_S = f$. By using Theorem \ref{division theorem for multiplier ideal sheaf},
we can find a solution $(F_t,G_t)$ such that $ F_t+G_t = \tilde{F}$ and 
$$\int_{Y} c(\Psi_t) \left(\frac{1+t}{t}\abs{G_t}^2 e^{-\Psi} + (1+t)\abs{F_t}^2\right) e^{-\phi} \le \int_{Y} \left( c(\Psi_t) +\frac{D({\Psi_t})}{(\frac{1}{1+t} + \frac{t}{1+t} e^{\Psi})}\right) \abs{\tilde{F}}^2 e^{-\phi}.
$$   Since $\abs{G}^2 \abs{T}^{-2k}$ is locally integrable, we have that $G$ is vanish on $S$. Thus we can find extension $F_t$ with the estimate

$$\int_{Y} c(\Psi_t) \abs{F_t}^2 e^{-\phi} \le \int_{Y} \frac{1}{1+t}\left( c(\Psi_t) +\frac{D({\Psi_t})}{(\frac{1}{1+t} + \frac{t}{1+t} e^{\Psi})}\right) \abs{\tilde{F}}^2 e^{-\phi}.
$$
Since $Y$ is relatively compact, we also have that $c(\Psi_t)\abs{\tilde{F}}^2 e^{-\phi}$ is upper bounded. Thus by Lemma \ref{calculus lemma}, we get a extension with the estimate we seek, by letting $t \rightarrow +\infty$. When $\phi$ is not continuious, one can consider the decreasing sequence of continuious plurisubharmonic functions ${\phi_j}$ on $Y$ (Such a sequence can always be taken, since $X$ is stein), and take the limit $i \rightarrow \infty$.
To obtain a solution on $X$, we take an increasing sequence of relayively compact stein sub-domain $X_j$ converging to $X$. It suffice to apply the above argument to $X_j$.
\end{proof}
\end{section}
\begin{section}{Another characterization of plurisubharmonicity via the sharp $L^2$-division}
In \cite{DWZ18}, they show that the weight function $e^{-\phi}$ can yield the sharp $L^2$-extension theorem only if $\phi$ is plurisubharmonic. By observing the proof of Theorem C in Section \ref{A proof of sharp L2 extension theorem}, we are led to define an another $L^2$- division property that also characterizes plurisubharmonicity.\\
The following lemma is usefull criterion for verifying plurisubharmonicity.
    \begin{lem}[\cite{Ka22}]\label{mean value inequality}
        Let $\Omega \subset \mathbb{C}^n$ be a domain, and $\phi$ be a upper semicontinuious and locally integrable function on $\Omega$. If for any $x \in \Omega$ and $x+E(A) \subset \Omega$, the mean value inequality 
        \begin{equation}\label{mean value inequality for ellipsoids}
            \phi(x) \le \frac{1}{\vol{(E(A))}} \int_{x+E(A)} \phi(y) d\lambda(y)
        \end{equation}
        holds, then $\phi$ is plurisubharmonic.
    \end{lem}   

        Now we recall how the sharp $L^2$-extension property characterizes plurisubharmonicity (see Definition\ref{L2 properties} or \cite{DNWZ23}). In \cite{DNWZ23} they defined the $L^2$-property using holomorphic cylinders, but the same result can be obtained using an ellipsoids instead of holomorphic cylinders by Lemma \ref{mean value inequality}.
        Assume that for any ellipsoid $x+E(A)$, there exists holomorphic function satisfying $f(x) = 1$ and $$\int_{x+E(A)}\abs{f}^2 e^{-\phi} \frac{d\lambda}{\vol(E(A))} \le e^{-\phi(x)}.$$
        By the Young inequality and plurisubharmonicity of $\log{\abs{f}^2}$, we obtain the inequality
        $$ \int_{x+E(A)} \phi(x) \frac{d\lambda}{\vol(E(A))} \ge \int_{x+E(A)} (\phi(x) - \log{\abs{f}^2}) \frac{d\lambda}{\vol{E(A)}} \ge \phi(x).$$
        By Lemma \ref{mean value inequality}, it turns out that $\phi$ is plurisubharmonic.\\
        \begin{proof}[Proof of Theorem D]
        If $\phi$ satisfies the condition in Theorem D, then we can show that $\phi$ has sharp $L^2$-extension property for ellipsoids as a same way in the proof of Theorem C by letting $\theta \rightarrow +0$. It implies that $\phi$ is plurisubharmonic by the above argument.
    \end{proof}
  \end{section}
\section*{Acknowledgement}
  I am deeply grateful to Professor Hisamoto for offering valuable advice during the preparation of this manuscript.
I also thank my parents for their steady support throughout my studies in mathematics.

\end{document}